\documentclass[12pt]{amsart}
\input xy  \xyoption{all}

\DeclareMathOperator\K{\mathbb K}
\DeclareMathOperator\LL{\mathbb L}
\DeclareMathOperator\A{\mathbb A}
\DeclareMathOperator\B{\mathbb B}
\DeclareMathOperator\X{\mathbb X}
\DeclareMathOperator\Y{\mathbb Y}
\DeclareMathOperator\N{\mathbb N}
\DeclareMathOperator\QQ{\mathcal Q}
\DeclareMathOperator\E{\mathcal E}
\DeclareMathOperator\M{\mathcal M}
\DeclareMathOperator\F{\mathcal F}

\DeclareMathOperator\C{\mathbb C}
\DeclareMathOperator\Z{\mathbb Z}

\DeclareMathOperator\Gr{{\mathrm{Gr}}}

\DeclareMathOperator\I{\mathcal I}
\DeclareMathOperator\ii{{\boldsymbol{i}}}
\DeclareMathOperator\rr{{\boldsymbol{r}}}
\DeclareMathOperator{\rk}{rk}

\DeclareMathOperator{\Hom}{Hom}
\DeclareMathOperator{\DDelta}{\boldsymbol{\Delta}}
\DeclareMathOperator{\CC}{\boldsymbol{C}}

\newtheorem{fact}{Fact}[section]
\newtheorem{lemma}[fact]{Lemma}
\newtheorem{theorem}[fact]{Theorem}
\newtheorem{definition}[fact]{Definition}
\newtheorem{example}[fact]{Example}
\newtheorem{rremark}[fact]{Remark}
\newenvironment{remark}{\begin{rremark} \rm}{\end{rremark}}
\newtheorem{proposition}[fact]{Proposition}

\newtheorem{conjecture}[fact]{Conjecture}

\title{Quiver polynomials in iterated residue form}

\author{R. Rim\'anyi}
\address{Department of Mathematics, University of North Carolina at Chapel Hill, USA}
\address{Department of Mathematics, University of Geneva, Switzerland}
\email{rimanyi@email.unc.edu}

\begin{document}

\begin{abstract}
Degeneracy loci polynomials for quiver representations generalize several important polynomials in algebraic combinatorics. In this paper we give a nonconventional generating sequence description of these polynomials, when the quiver is of Dynkin type.
\end{abstract}

\maketitle

\section{Introduction}

Associate nonnegative integers to the vertices of an oriented graph. This combinatorial data determines a so-called quiver representation of a group $\Gamma$ on a vector space $V$. If the underlying unoriented graph is of simply laced Dynkin type, then the quiver representation has finitely many orbits. Let $\Omega$ be one of these orbits. The main object of this paper is the equivariant fundamental class
$$[\overline{\Omega}]\in H^*_\Gamma(V).$$
The ring $H^*_\Gamma(V)$ is a polynomial ring, hence one calls the class $[\overline{\Omega}]$ a quiver polynomial.

\smallskip

{\em Quiver polynomials are universal polynomials expressing degeneracy loci.} Tracing back the definitions of equivariant cohomology we obtain that quiver polynomials are universal polynomials representing degeneracy loci. We will illustrate this statement with the following example. Let $\E_1$ and $\E_2$ be bundles of rank $e_1$ and $e_2$ over the compact complex manifold $X$, and let $\phi$ be a bundle map $\E_1\to \E_2$ which satisfies a certain transversality condition. Let $r\leq e_1\leq e_2$, and define the degeneracy locus
$$\Omega_r(\phi)=\{x\in X: \rk(\phi_x: (\E_1)_x \to (\E_2)_x )\leq e_1-r.\}$$
The Giambelli-Thom-Porteous formula claims that the fundamental cohomology class $[\Omega_r(\phi)]\in H^*(X)$ can be expressed as $\det\left(c_{r+j-i}(\E_1^*-\E_2^*)\right)_{i,j=1,\ldots,e_2-e_1+r}$.
The polynomial $\det\left(c_{r+j-i}\right)_{i,j}$
is a quiver polynomial. It only depends on the combinatorial data $\bullet \to \bullet$ (the graph), $e_1, e_2$ (the ranks of the bundles), $r$ (the orbit). Nevertheless it has the universality property that substituting $c_i(\E_1^*-\E_2^*)$ into the variables $c_i$, the polynomial expresses the fundamental class of the degeneracy locus $\Omega_r(\phi)$.

Quiver polynomials are analogous universal polynomials for more complicated degeneracy loci. This more complicated setting involves more than two vector bundles over $X$ (the vertices of the graph), with some bundle maps among them (the edges of the graph). The degeneracy locus is then defined by degenerations of a diagram of linear maps between vector spaces.

\smallskip

{\em History.} Quiver polynomials of Dynkin type generalize several important polynomials in Algebraic Combinatorics, for example
the Giambelli-Thom-Porteous formulae \cite{porteous}, the double Schur and Schubert polynomials of Schubert calculus \cite{fulton_universal}, and the
quantum \cite{FGP} and universal Schubert polynomials \cite{fulton_universal}.

There has been a lot of activity in the past decade or so, to find various formulae and/or algorithms to calculate quiver polynomials of certain Dynkin types, such as \cite{B:Gr, buch, bfr, buch-fulton,bkty, br, BSY, FR:quiver, FGP, KS, KMS}. By now we can claim that there are effective methods to find any such quiver polynomial. The three papers that approach the problem for general quivers of Dynkin type A,D,E are \cite{FR:quiver}, \cite{KS}, and \cite{buch}.

\smallskip

{\em The structure of quiver polynomials.} The key challenge of the theory is understanding the structure of quiver polynomials. The following two phenomena have been discovered/conjectured:
\begin{itemize}
\item{} Stability. In \cite{structure} the authors study an analogous problem: Thom polynomials of singularities. They find that equivariant fundamental classes display unexpected stability properties \cite[Theorems 2.1, 2.4]{structure}, which enables them to organize infinitely many such classes into a (nonconventional) rational generating sequence. This phenomenon is further developed and organized in \cite{bsz}, \cite{ts}, \cite{kaza} (under the name of Iterated Residue generating sequence). Theorem 2.1 of \cite{structure} applies to quiver representations as well.
\item{} Positivity. In \cite{buch} Buch proved that quiver polynomials are linear combinations of certain products of Schur polynomials. He conjectured that the coefficients in this expression are all non-negative. These coefficients provide a wide generalization of several combinatorial constants in the style of Littlewood-Richardson coefficients. The usual techniques of proving positivity in equivariant cohomology (such as geometric intersection numbers, Gr\"obner degenerations, interpolation theory, counting arguments) so far failed to prove Buch's conjecture.
\end{itemize}
In this paper we prove an Iterated Residue description of quiver polynomials.
 This description reduces Buch's conjecture to a positivity conjecture on the coefficients of certain expansions of rational functions. R.~Kaliszewski \cite{ryan_phd}  has promising initial results towards proving this algebraic positivity conjecture.

\bigskip

The author thanks J. Allman, A. S. Buch, L. Feh\'er, R. Kaliszewski, M. Kazarian, and A.~Szenes for valuable discussions on the topic, as well as the hospitality of MPIM Bonn and the University of Geneva. The author was partially supported by NSF grant DMS-1200685.

\section{Combinatorics of polynomials}

\subsection{$\Delta$-polynomials, Schur polynomials} \label{sec:schur}

Let $c_n$ be elements of a ring for $n\in \Z$. Assume that $c_0=1$ and $c_n=0$ for $n<0$.
Write $\K=(c_n)_{n\in \Z}$. Let $\lambda=(\lambda_1,\ldots,\lambda_r)\in \Z^r$ be an integer vector.

\begin{definition}
Define the $\Delta$-polynomial
$$\Delta_\lambda=\Delta_\lambda(\K)=\det\left( c_{\lambda_i+j-i} \right)_{i,j=1,\ldots,r}.$$\
\end{definition}

If $\lambda$ is a partition, ie. if $\lambda_1\geq \lambda_2 \geq \ldots \geq \lambda_r$, then we call a $\Delta$-polynomial a {\em Schur polynomial}. It is known that if the underlying ring is $\Z[c_1,c_2,\ldots]$, then the Schur polynomials form an additive basis of this ring.

When we want to emphasize that a $\Delta$-polynomial is not a Schur polynomial, then we call it a {\em fake} Schur polynomial. Observe that a fake Schur polynomial is either 0, or is equal to plus or minus a Schur polynomial. For example $\Delta_{34}=0$, $\Delta_{14}=-\Delta_{32}$, $\Delta_{154}=\Delta_{433}$.

\subsection{The $\CC$ and $\DDelta$-operations}

Let $p$ be a positive integer. For all $k=1,\ldots,p$ let $\K_k$ $=$ $(c_{kn})_{n\in \Z}$ be a set of elements in a ring, as in Section \ref{sec:schur}. In particular, for all $k$ we have $c_{k0}=1$ and $c_{kn}=0$ for $n<0$.

Let $\A_1,\ldots,\A_p$ ($|\A_k|=r_k$)  be disjoint finite ordered sets of variables: $\A_k=\{u_{k1},\ldots,u_{kr_k}\}$.

\begin{definition}
For a Laurent monomial in the variables $\cup_k \A_k$ define
\[
\CC
\left(\prod_{k=1}^p \prod_{s=1}^{r_k} u_{ks}^{\lambda_{ks}}
 \right)=
\CC_{\A_1,\ldots,\A_p}^{\K_1,\ldots,\K_p}
\left(
\prod_{k=1}^p \prod_{s=1}^{r_k} u_{ks}^{\lambda_{ks}}
 \right)
 =
\prod_{k=1}^p \prod_{s=1}^{r_k} c_{k\lambda_{ks}}.
\]
For an element of $\Z[[u_{ks}^{\pm1}]]$, which has finitely many monomials with non-0 $\CC$-value, extend this operation linearly.
\end{definition}

\begin{definition}
For a Laurent monomial in the variables $\cup_k \A_k$ define
\[
\DDelta
\left(\prod_{k=1}^p \prod_{s=1}^{r_j} u_{ks}^{\lambda_{ks}}
 \right)=
\DDelta_{\A_1,\ldots,\A_p}^{\K_1,\ldots,\K_p}
\left(
\prod_{k=1}^p \prod_{s=1}^{r_j} u_{ks}^{\lambda_{ks}}
 \right)
 =
\prod_{k=1}^p \Delta_{\lambda_{k1},\ldots,\lambda_{kr_k}}(\K_k)
\]
For an element of $\Z[[u_{ks}^{\pm1}]]$, which has finitely many monomials with non-0 $\DDelta$-value, extend this operation linearly.
\end{definition}

\begin{example} \rm
Let $\A_1=\{u_1,u_2\}$, $\A_2=\{v_1,v_2\}$, $\A_3=\{w\}$, and let $\K_1=(c_n)$, $\K_2=\K_3=(d_n)$. Here are some examples for the $\CC$ and $\DDelta$-operations.
\begin{itemize}
\item{} $
\CC(u_1^2 u_2^1 v_1^5 v_2^1 w^2 + u_1^2 u_2^1 v_1^1 v_2^2 w^3
        + u_1^1 u_2^4 +  v_1^{-1} v_2^2 + u_1^{-1} u_2^{-2})=$ \hfill \

\ \hfill $c_2c_1d_5d_1d_2 + c_2c_1d_1d_2d_3+ c_1c_4+\underbrace{d_{-1}}_{0}d_2+\underbrace{c_{-1}}_{0} \underbrace{c_{-2}}_{0}.$
\item{} $
\CC(u_1^2u_2^2(u_1-u_2))=\CC(u_1^3u_2^2-u_1^2u_2^3)=c_3c_2-c_2c_3=0.$
\item{}
$\CC\left(\frac{u_1u_2}{1-\frac{u_1}{u_2}}\right)=\CC(u_1u_2+u_1^2+u_1^3u_2^{-1}+\ldots)=c_1^2+c_2.$
\item{}
$\DDelta(u_1^2 u_2^1 v_1^5 v_2^1 w^2 + u_1^2 u_2^1 v_1^1 v_2^2 w^3
        + u_1^1 u_2^4 +  v_1^{-1} v_2^2 + u_1^{-1} u_2^{-2})= \qquad\qquad\qquad\qquad\qquad\qquad\qquad
$
$\qquad \Delta_{21}(\K_1)\Delta_{51}(\K_2)\Delta_2(\K_2)+ \Delta_{21}(\K_1)\underbrace{\Delta_{12}(\K_2)}_0\Delta_3(\K_2) +
\underbrace{\Delta_{14}(\K_1)}_{-\Delta_{32}(\K_1)} + \underbrace{\Delta_{-1,2}(\K_2)}_{-\Delta_{10}(\K_2)} + \underbrace{\Delta_{-1,-2}(\K_1)}_0.$
\item{}
$\DDelta\left(\frac{u_1u_2}{1-\frac{u_1}{u_2}}\right)=\DDelta(u_1u_2+u_1^2+u_1^3u_2^{-1}+\ldots)=
\Delta_{11}(\K_1)+\Delta_{20}(\K_1)+ \underbrace{ \Delta_{3,-1}(\K_1)+\ldots}_0 .$
\end{itemize}
Here we used a convention, that we will keep using in the whole paper: by the rational function $1/(1-x/y)$ we always mean its expansion in the $|x|\ll |y|$ range, that is, $1+(x/y)+(x/y)^2+\ldots$.
\end{example}

The two defined operations are related with each other.

\begin{lemma} \label{lemma:C vs Delta}
Let $\A=\{u_1,\ldots, u_{r}\}$, and let $f$ be a function in these variables. Then
\[
\CC_{\A}^{\K}\left(f \cdot  \prod_{i<j} \left(1-\frac{u_i}{u_j}\right) \right) =
\DDelta_{\A}^{\K} \left( f \right).
\]
\end{lemma}

The proof is an application of simple properties of the determinant and the discriminant $\prod_{i<j} (1-u_i/u_j)$, illustrated by the following special case:
\[ \CC\left( (u_1^8u_2^5)(1-u_1/u_2)\right) = \CC\left( u_1^8u_2^5 - u_1^9u_2^4\right)=
c_8c_5-c_9c_4= \Delta_{85} = \DDelta(u_1^8 u_2^5).\]
We leave details to the reader.

\subsection{Constant-term formula}

The following lemma is a reformulation of the definition of $\CC$.

\begin{lemma} \label{lemma:C_vs_const}
Let $\K_k=(c_{kn})$, $\A_k=\{u_{ks}\}$. Then
\[
\CC_{\A_1,\ldots,\A_p}^{\K_1,\ldots,\K_p}(f(u_{ks})) =
\left\{ f(u_{ks}) \cdot \prod_{k,s} \sum_{n=0}^\infty \frac{c_{kn}}{ u_{ks}^n} \right\}_{u_{ks}},
\]
where $\{\ \}_{u_{ks}}$ means the constant coefficient in the $u_{ks}$ variables.
\end{lemma}

\begin{remark}
Taking the constant term (after shifting of exponents) is an iterated residue operation. Hence, we call $\CC$ and $\DDelta$ Iterated Residue operations, c.f. \cite{bsz}, \cite[Sect.11]{ts}, \cite{kaza}.
\end{remark}

\section{Quivers of bundles, and their characteristic classes}
Let $\QQ=(\QQ_0,\QQ_1)$ be an oriented graph (the quiver), with the set $\QQ_0=\{1,2,\ldots,N\}$ of vertexes, and a finite set $\QQ_1$ of arrows. An arrow $a\in \QQ_1$ has a tail $t(a)\in \QQ_0$ and head $h(a)\in \QQ_0$. For $i\in \QQ_0$ define
$$T(i)=\{j\in \QQ_0| \exists a\in \QQ_1 \text{ with }t(a)=j, h(a)=i\},$$
$$H(i)=\{j\in \QQ_0| \exists a\in \QQ_1 \text{ with }h(a)=j, t(a)=i\}.$$

By a $\QQ$-bundle over a space $X$ we mean a collection of vector bundles $\E_i$, $i\in \QQ_0$.
A characteristic class of a $\QQ$-bundle is a polynomial in the Chern classes of the bundles $\E_1,\ldots,\E_n$.

Let $\M_i=\oplus_{j \in T(i)} \E_j$, and consider the virtual bundle $\M^*_i-\E^*_i$, where ${}^*$ means the dual bundle. Its Chern classes are defined by the formal expansion
$$\sum_{n=0}^\infty c_n(\M^*_i-\E^*_i)\xi^n = \frac{\sum_{n=0}^{\rk \M_i} c_n(\M_i) (-\xi)^n}{\sum_{n=0}^{\rk \E_i\ \ \  } c_n(\E_i) (-\xi)^n}=
\frac{\prod_{j\in T(i)} \left( \sum_{n=0}^{\rk \E_j} c_n(\E_j) (-\xi)^n \right)}
{\sum_{n=0}^{\rk \E_i} c_n(\E_i) (-\xi)^n}.$$

Let $\K_i=( c_n(\M^*_i-\E^*_i))_n$. We will be concerned with characteristic classes of a $\QQ$-bundle that can be written as polynomials in $\cup \K_i$. For example, suppose we have sets of variables
$$\A_1=\{u_{11},\ldots,u_{1r_1}\}, \ldots, \A_p=\{u_{p1},\ldots,u_{pr_p}\},$$
and numbers $i_1,\ldots, i_p\in \QQ_0$. Then for any polynomial (or Laurent polynomial) $f$ in the variables $u_{ks}$, the expressions
\begin{equation}\label{eqn:ir_ex}
\CC_{\A_1,\ldots,\A_p}^{\K_{i_1},\ldots,\K_{i_p}} (f), \qquad\qquad \DDelta_{\A_1,\ldots,\A_p}^{\K_{i_1},\ldots,\K_{i_p}} (f)
\end{equation}
are characteristic classes of the $\QQ$-bundle.

\section{Quiver representations, quiver polynomials}

\subsection{Quiver representations, and their orbits} \label{sec:orbits}

Recall that $\QQ=(\QQ_0,\QQ_1)$ is an oriented graph (the quiver), with the set $\QQ_0=\{1,2,\ldots,N\}$ of vertices, and a finite set $\QQ_1$ of arrows $a=(t(a),h(a))\in \QQ_0^2$. Fix a dimension vector $(e_1,\ldots,e_N)\in \N^N$, and vector spaces $E_i=\C^{e_i}$. We define the quiver representation of the group $\Gamma=\times_{i=1}^N GL(E_i)$ on the vector space $V=\oplus_{a\in \QQ_1} \Hom(E_{t(a)},E_{h(a)})$ by
$$\left( g_i \right)_{i \in \QQ_0} \cdot \left( \phi_a \right)_{a\in \QQ_1} = \left(  g_{h(a)} \circ \phi_a \circ g_{t(a)}^{-1}\right)_{a\in \QQ_1}.$$

Although some results of the paper are valid for more general quivers, we will restrict our attention to the Dynkin quivers, that is, to graphs $\QQ$ whose underlying unoriented graph is one of the simply-laced Dynkin graphs $A_n$, $D_n$, $E_6$, $E_7$, $E_8$. For these graphs the quiver representation has finitely many orbits \cite{bgp}---for any orientation and any dimension vector. Moreover, the orbits have an explicit description, as follows.
Consider the set $\Phi^+$ of positive roots, and the set $\{\alpha_i:i=1,\ldots,N\}$ of simple roots for the corresponding root system. For a positive root $\alpha$ define $d_i(\alpha)$ by $\alpha=\sum_{i=1}^N d_i(\alpha) \alpha_i$. The orbits of the quiver representation with Dynkin graph $\QQ$ and dimension vector $(e_1,\ldots,e_N)\in \N^N$ are in one to one correspondence with vectors
\begin{equation}\label{eqn:orbit_m}
(m_\alpha)\in \N^{\Phi^+},\qquad \text{for which}\qquad \sum_{\alpha \in \Phi^+} m_\alpha d_i(\alpha) = e_i \ \text{for all}\ i=1,\ldots,N.
\end{equation}
Note that the list of orbits does not depend on the orientation of $\QQ$. The orbit corresponding to the vector $m=(m_\alpha)\in \N^{\Phi^+}$ will be denoted by $\Omega_m$.

\subsection{Quiver polynomials} \label{sec:poly}

The main object of this paper is the $\Gamma$-equivariant cohomology class represented by the orbit closure $\overline{\Omega}_m$ in $V=\oplus_{a\in \QQ_1} \Hom(E_{t(a)},E_{h(a)})$ i.e.
$$[\overline{\Omega}_m]\in H^*_\Gamma(V).$$
There are several equivalent ways to define the equivariant class represented by an invariant subvariety in a representation, see for example \cite{kazass}, \cite{edidin_graham}, \cite{cr}, \cite[8.5]{miller-sturmfels}, \cite{fultonnotes}. Formally we follow the treatment in \cite{buch}, which defines the K-theory class of the subvariety, and defines the  cohomology class as a certain leading coefficient.

\medskip

The ring $H^*_\Gamma(V)$ is the ordinary cohomology ring $H^*(B_\Gamma V)$ of the so-called Borel-construction space $B_\Gamma V$. There is a natural vector bundle $B_\Gamma V \to B\Gamma$ with fiber $V$; where $B\Gamma$ is the classifying space of $\Gamma$. Hence $H^*_\Gamma(V)\cong H^*(B\Gamma)$, the ring of $\Gamma$-characteristic classes. Recall that $B\Gamma=BGL(E_1)\times \ldots \times BGL(E_n)$ and that over each $BGL(E_i)$ there is a natural tautological vector bundle $\E_i$ with fiber $E_i$. Consider the bundles $\E_i$ over $B\Gamma$. This way we obtained a $\QQ$-bundle over $B\Gamma$. Recall that $\M_i=\oplus_{j\in T(i)} \E_i$, and define
\begin{equation} \label{eqn:Ki}
\K_i=\left(c_n(\M^*_i-\E^*_i)\right)_n
\end{equation}
Our goal is to express the class $[\overline{\Omega}_m]$ in terms of the characteristic classes $\cup \K_i$.


\section{The main result} \label{sec:main_result}

\subsection{Resolution pair} \label{sec:resolution_pair}
First we follow Reineke \cite{reineke} to define {\sl resolution pairs}, see also \cite{buch}. For dimension vectors $e,f \in \N^N$ let $\langle e,f\rangle=\sum_{i=1}^N e_i f_i-\sum_{a\in \QQ_1} e_{t(a)}f_{h(a)}$ denote the Euler form for $\QQ$.
Identifying a positive root $\alpha$ with its dimension vector $d(\alpha)\in \N^N$ by $\alpha=\sum_{i=1}^N d_i(\alpha) \alpha_i$ (where $\alpha_i$ are the simple roots), we can extend the Euler form for positive roots.

Let $\Phi'\subset \Phi^+$ be a subset of the set of positive roots.  A partition
\begin{equation}\label{eqn:partition_pos_roots}
\Phi'=\I_1 \cup \I_2 \cup \ldots \cup \I_s
\end{equation}
is called directed if $\langle \alpha,\beta\rangle\geq 0$ for all $\alpha, \beta\in \I_j$ $(1\leq j \leq n)$, and $\langle\alpha,\beta\rangle\geq 0 \geq \langle\beta,\alpha\rangle$ for all $\alpha\in \I_i$, $\beta\in \I_j$ with $1\leq i<j \leq s$. A directed partition exists for any Dynkin quiver, see \cite{reineke}.

Let $m=(m_\alpha)\in\N^{\Phi^+}$ be a vector representing an orbit $\Omega_m\subset V$, let $\Phi'\subset \Phi^+$ be a subset containing $\{\alpha: m_\alpha \not= 0\}$, and let $\Phi'=\I_1\cup \ldots \cup \I_s$ be a directed partition. For each $j\in \{1,\ldots,s\}$ write
$$\sum_{\alpha \in \I_j}  m_\alpha \alpha =(p^{(j)}_1,\ldots,p^{(j)}_n)\in \N^N.$$
Let $\ii^{(j)}=(i_1,\ldots,i_l)$ be any sequence of the vertices $i\in \QQ_0$ for which $p^{(j)}_i\not= 0$, with no vertices repeated, and ordered so that the tail of any arrow of $\QQ$ comes before its head. Let $\rr^{(j)}=(p^{(j)}_{i_1},\ldots,p^{(j)}_{i_l})$. Finally, let $\ii$ and $\rr$ be the concatenated sequences $\ii=\ii^{(1)}\ii^{(2)}\ldots  \ii^{(s)}$ and $\rr=\rr^{(1)}\rr^{(2)}\ldots \rr^{(s)}$. The pair $\ii, \rr$ will be called a {\em resolution pair} for $\Omega_m$.

\subsection{Resolution variables, generating functions}

For the resolution pair
$$\ii=(i_1,\ldots,i_p),\ \qquad \rr=(r_1,\ldots, r_p)$$
we define the sets of variables $\A_k=\{u_{k1},\ldots,u_{kr_k}\}$ for $k=1,\ldots,p$. Set
\[
\B_k=\mathop{\bigcup_{l>k}}_{i_l=i_k} \A_l, \qquad\qquad
\C_k=\mathop{\bigcup_{l>k}}_{i_l\in H(i_k)} \A_l, \qquad\qquad
n_k=\mathop{\sum_{l>k}}_{i_l\in T(i_k)} r_l -  \mathop{\sum_{l>k}}_{i_l=i_k} r_l.
\]

Next we define various functions in the variables $\cup_k \A_k$. For $1\leq k \leq p$ define
\begin{itemize}
\item the monomial factors
$$M_k=\prod_{s=1}^{r_k}  u_{ks}^{n_k};$$
\item the discriminant factors
$$D_k= \prod_{1\leq i<j \leq r_k} \left(1-\frac{u_{ki}}{u_{kj}}\right);$$
\item and the interference factors
$$I_k= \prod_{s=1}^{r_k}
\frac{ \prod_{x\in \B_k } \left( 1- \frac{u_{ks}}{ x } \right) }{ \prod_{x\in \C_k } \left( 1- \frac{u_{ks}}{ x } \right) }.
$$
\end{itemize}

\begin{theorem} \label{thm:main}
Recall the definition of $\K_i$ from (\ref{eqn:Ki}). We have
\begin{equation} \label{eqn:main1}
[\overline{\Omega}_m]=\CC_{\A_1,\ldots,\A_p}^{\K_{i_1},\ldots,\K_{i_p}} \left( \prod_{k=1}^p M_k D_k I_k \right);
\end{equation}
\begin{equation}\label{eqn:main2}
[\overline{\Omega}_m]=\DDelta_{\A_1,\ldots,\A_p}^{\K_{i_1},\ldots,\K_{i_p}} \left( \prod_{k=1}^p {M}_k I_k \right).
\end{equation}
\end{theorem}

The right hand sides of (\ref{eqn:main1}) and (\ref{eqn:main2}) are equal, because of Lemma \ref{lemma:C vs Delta}. We will prove (\ref{eqn:main2}) in Section \ref{sec:proof}.

\begin{remark} The key ingredient of the above expressions is $\prod_{k} I_k$. It can be rewritten as
\begin{equation}\label{eqn:COHAmult}
\prod_{k=1}^p I_k=  \frac{ \prod_{v\in Q_0} \prod_{k<l, i_k=i_l=v}\ \ \ \ \ \ \  \prod_{x\in \A_k, y\in \A_l} \left( 1 - \frac{x}{y} \right) }
{ \prod_{a\in Q_1} \prod_{k<l, i_k=t(a), i_l=h(a)} \prod_{x\in \A_k, y\in \A_l} \left( 1 - \frac{x}{y} \right) }.
\end{equation}
The resemblance of this formula with the one in Theorem 2 of \cite{coha} for the multiplication in the Cohomological Hall Algebra, is not accidental. The connection between COHA's and quiver polynomials is explained in \cite{rr_coha}.
\end{remark}

\begin{remark} \label{rem:nonumer}
Lemma \ref{lemma:C vs Delta} allows us to write (\ref{eqn:main2}) in another equivalent form, as follows. Let $\X_i$ be the concatenation of alphabets $\A_{j_1}\A_{j_2}\ldots\A_{j_s}$ with $j_1< j_2<\ldots <j_s$ and $i_{j_1}=\ldots=i_{j_s}=i$. Then
\begin{equation}\label{eqn:main3}
[\overline{\Omega}_m]=\DDelta_{\X_1,\ldots,\X_N}^{\K_{1},\ldots,\K_{N}} \left(
\frac{\prod_{k=1}^p {M}_k}
{\prod_{a\in Q_1} \prod_{k<l, i_k=t(a), i_l=h(a)} \prod_{x\in \A_k, y\in \A_l} \left( 1 - \frac{x}{y} \right) }
\right).
\end{equation}
\end{remark}

\subsection{An $A_3$ example}

Consider the quiver $1 \to 2 \leftarrow 3$. The positive roots of the root system $A_3$ are $\alpha_{ij}=\sum_{i\leq u\leq j}\alpha_u$ for $1\leq i\leq j \leq 3$. Consider the orbit corresponding to the linear combination $\sum_{1\leq i\leq j\leq 3} m_{ij}\alpha_{ij}$.
Let us choose the partition
$$\Phi^+=\{\alpha_{22}\} \cup \{\alpha_{12},\alpha_{23},\alpha_{13}\} \cup \{\alpha_{11},\alpha_{33}\}.$$
This choice induces the resolution pair $\ii=(2,1,3,2,1,3)$, $\rr=(m_{22},m_{12}+m_{13}, m_{23}+m_{13}, m_{12}+m_{13}+m_{23},m_{11}, m_{33})$.
Let $\A_i, i=1,\ldots,6$ be alphabets of cardinalities $\rr_i$ respectively. Theorem~\ref{thm:main} gives
\begin{equation*}
[\overline{\Omega}_m]=\DDelta_{\A_1,\ldots,\A_6}^{\K_{2},\K_{1},\K_{3},\K_{2},\K_{1},\K_{3}}
\left(
\A_1^{m_{13}+m_{11}+m_{33}}\A_2^{-m_{11}}\A_3^{-m_{33}}\A_4^{m_{11}+m_{33}}
\frac{ \left(1-\frac{\A_1}{\A_4}\right) \left(1-\frac{\A_2}{\A_5}\right)\left(1-\frac{\A_3}{\A_6}\right)}
{\left(1-\frac{\A_2}{\A_4}\right)\left(1-\frac{\A_3}{\A_4}\right)}
\right).
\end{equation*}
Here we used the shorthand notations $\X^m=\prod_{x\in \X} x^m$ and $1-\X/\Y=\prod_{x\in \X}\prod_{y\in \Y} (1-x/y)$. This formula can be rewritten in the form of Remark \ref{rem:nonumer} to obtain
\begin{equation*}
[\overline{\Omega}_m]=\DDelta_{(\A_1\A_4),(\A_2\A_5),(\A_3\A_6)}^{\K_{2},\K_{1},\K_{3}}
\left(
\frac{\A_1^{m_{13}+m_{11}+m_{33}}\A_2^{-m_{11}}\A_3^{-m_{33}}\A_4^{m_{11}+m_{33}}}
{\left(1-\frac{\A_2}{\A_4}\right)\left(1-\frac{\A_3}{\A_4}\right)}
\right).
\end{equation*}
It is tempting to call the functions in the arguments ``generating functions'': One only has to change the exponents $m_{ij}$ to obtain the polynomials $[\overline{\Omega}_m]$ for different $m$ values. Nevertheless, the lengths of the alphabets $\A_k$ also depend on $m_{ij}$, and the operations $\CC$, $\DDelta$ are rather delicate. Hence, one may call the arguments ``iterated residue generating functions''.

\section{Concrete calculations} \label{sec:examples}

\subsection{Examples for (\ref{eqn:main1})}

Consider the quiver $1 \to 2 \leftarrow 3$. The positive roots of the root system $A_3$ are $\alpha_{ij}=\sum_{i\leq u\leq j}\alpha_u$ for $1\leq i\leq j \leq 3$. Consider the orbit corresponding to the linear combination $\sum_{1\leq i\leq j\leq 3} m_{ij}\alpha_{ij}$. We will show an example with two different choices of (\ref{eqn:partition_pos_roots}), that is two different partitions of the set of positive roots.

First, we may choose the partition
$$\Phi^+=\{\alpha_{22}\} \cup \{\alpha_{12},\alpha_{23},\alpha_{13}\} \cup \{\alpha_{11},\alpha_{33}\}.$$
This choice induces the resolution pair $\ii=(2,1,3,2,1,3)$, $\rr=(m_{22},m_{12}+m_{13}, m_{23}+m_{13}, m_{12}+m_{13}+m_{23},m_{11}, m_{33})$.
For the special case $m_{13}=m_{12}=m_{22}=m_{33}=1$ and all other $m_{ij}=0$  we obtain the dimension vector $(2,3,2)$ and have the variables
$$\A_1=\{u_{11}\}, \A_2=\{u_{21}, u_{22}\}, \A_3=\{u_{31}\}, \A_4=\{u_{41},u_{42}\}, \A_5=\{\}, \A_6=\{u_{61}\}$$
which we rename $u, v_1, v_2, w, s_1, s_2, x$ respectively. We obtain that
\begin{equation}\label{eqn:G1}
G_1=\prod_{k=1}^6 M_kD_kI_k=\frac{u^2s_1s_2(1-\frac{v_1}{v_2})(1-\frac{s_1}{s_2})(1- \frac{u}{s_1})(1- \frac{u}{s_2})(1- \frac{w}{x})}
{w(1- \frac{v_1}{s_1})(1- \frac{v_2}{s_1})(1- \frac{v_1}{s_2})(1- \frac{v_2}{s_2})
(1- \frac{w}{s_1})(1- \frac{w}{s_2})}.
\end{equation}
Calculation shows that this function is equal to $-u^3+u^2s_2+u^2v_2+$fractions. Hence the corresponding quiver polynomial is
\begin{equation}\label{eqn:quiver_pol_example}
[\overline{\Omega}_m]=-c_3^{(2)}+ c_2^{(2)}c_1^{(2)} + c_2^{(2)} c_1^{(1)},
\end{equation}
where the upper index ${}^{(i)}$ refers to the bundle $\M_i^*-\E_i^*$.

Now consider the same orbit, that is, the same $m$, but with the different choice of partition
$$\Phi^+=\{\alpha_{22},\alpha_{12},\alpha_{23}\} \cup \{\alpha_{13},\alpha_{11},\alpha_{33}\}.$$
The resolution pair obtained is $\ii=(1,3,2,1,3,2)$, $\rr=(m_{12}, m_{23}, m_{12}+m_{22}+m_{23}, m_{11}+m_{13}, m_{13}+m_{33}, m_{13})$.
After renaming the variables
$$\A_1=\{u_{11}\}, \A_2=\{\}, \A_3=\{u_{31},u_{32}\}, \A_4=\{u_{41}\}, \A_5=\{u_{51},u_{52}\}, \A_6=\{u_{61}\}$$
to $u,w_1,w_2,s,t_1,t_2,x$ respectively, we obtain
\begin{equation} \label{eqn:G2}
G_2=\prod_{k=1}^6 M_kD_kI_k = \frac{w_1^2w_2^2(1-\frac{w_1}{w_2})( 1 - \frac{u}{s}) (1-\frac{t_1}{t_2}) ( 1 - \frac{w_1}{x})( 1 - \frac{w_2}{x})}
{u ( 1 - \frac{u}{w_1})( 1 - \frac{u}{w_2})( 1 - \frac{s}{x})( 1 - \frac{u}{x})( 1 - \frac{t_1}{x})( 1 - \frac{t_2}{x})}.
\end{equation}
The functions (\ref{eqn:G1}) and (\ref{eqn:G2}) look similar, but they are essentially different. Let us show how much simpler the latter one is for practical purposes. Observe that in (\ref{eqn:G2}) the $x$ variable only turns up in the factor
$$\left( 1 - \frac{w_1}{x}\right)\left( 1 - \frac{w_2}{x}\right) \prod_{y\in \{s,u,t_1,t_2\}}\sum_{i=0}^{\infty} \left(\frac{y}{x}\right)^i=1+\frac{1}{x}\left( \ldots \right).$$
This means that the integer part (the sum of the terms without denominators) of (\ref{eqn:G2}) is the same as the integer part of (\ref{eqn:G2}) without this factor. We will denote this relation by the sign $\sim_x$. We obtain
$$G_2 \sim_x
\frac{w_1^2w_2^2(1-\frac{w_1}{w_2})( 1 - \frac{u}{s}) (1-\frac{t_1}{t_2})}
{u ( 1 - \frac{u}{w_1})( 1 - \frac{u}{w_2})} \sim_s
\frac{w_1^2w_2^2(1-\frac{w_1}{w_2}) (1-\frac{t_1}{t_2})}
{u  ( 1 - \frac{u}{w_1})( 1 - \frac{u}{w_2})} \sim_t
\frac{w_1^2w_2^2(1-\frac{w_1}{w_2})}{u ( 1- \frac{u}{w_1})( 1- \frac{u}{w_2})}.$$
The last rational function clearly has integer part $w_1w_2^2+w_2^2u-w_1^3$, showing again that the induced quiver polynomial is (\ref{eqn:quiver_pol_example}).

\begin{remark}
The formulas of Theorem \ref{thm:main} will be proved using a geometric construction, which is a series of rational maps. Some of these maps will happen to be the identity maps. This geometric phenomenon is reflected in the fact that the function $G_2$ can be simplified, as above, by removing all the factors containing variables $u_{ks}$ for certain $k$'s.
\end{remark}

\subsection{Examples for (\ref{eqn:main2})}
Consider again the quiver $Q=1 \to 2 \leftarrow 3$, with dimension vector $e=(2,3,2)$, and the orbit determined by $m_{13}=m_{12}=m_{22}=m_{33}=1$ and all other $m_{ij}=0$. Choose the partition
\[\Phi^+=\{ \alpha_{22} ,\alpha_{12}, \alpha_{23} \} \cup \{\alpha_{13}, \alpha_{11}, \alpha_{33}\}.
\]
The resolution pair obtained is $\ii=(1,3,2,1,3,2)$,
$$\rr=(m_{12},m_{23},m_{12}+m_{22}+m_{23}, m_{11}+m_{13},m_{13}+m_{33},m_{13})=(1, 0, 2, 1, 2, 1).$$
We have the variables
$$\A_1=\{u_{11}\}, \A_2=\{\}, \A_3=\{u_{31},u_{32}\}, \A_4=\{u_{41}\}, \A_5=\{u_{51},u_{52}\}, \A_6=\{u_{61}\},$$
which we rename to $u,w_1,w_2,s,t_1,t_2,x$ respectively. We have
$$
\begin{aligned}
\prod_{k=1}^6 M_kI_k & =\frac{ w_1^2w_2^2 }{u ( 1-\frac{u}{w_1} )( 1-\frac{u}{w_2} )} =w_1^2w_2+w_1w_2^2+\\ & u\left(w_1^2+w_1w_2+w_2^2\right)+u^2\left(\frac{w_2^2}{w_1}+w_1+w_2\right)+u^3\left(\frac{w_2}{w_1}+1\right)+u^4\left(\frac{1}{w_1}\right)+\ldots,
\end{aligned}
$$
where for the rest of the terms $u^aw_1^bw_2^c$ at least one of $a<0$, $b<-1$, or $c<0$ holds---hence, their $\DDelta$ is 0. From (\ref{eqn:main2}) of Theorem \ref{thm:main} we obtain
$$[\overline{\Omega}_m]=s^{(2)}_{2,1}+s^{(2)}_{1,2}+ s^{(1)}_1 \left(s^{(2)}_{2,0}+s^{(2)}_{1,1}+s^{(2)}_{0,2}\right) +
s^{(1)}_{2}\left( s^{(2)}_{-1,2}+s^{(2)}_{0,1}+s^{(2)}_{1,0}\right) +
s^{(1)}_{3}\left( s^{(2)}_{-1,1}+s^{(2)}_{0,0}\right)
s^{(1)}_{4}\left( s^{(2)}_{-1,0}\right).$$
Observing that $s^{(i)}_{a,a+1}=0$ for any $a$, as well as $s^{(i)}_{a,b}+s^{(i)}_{b-1,a+1}=0$ for any $a,b$, we obtain
$$[\overline{\Omega}_m]=s^{(2)}_{2,1}  +  s^{(1)}_1 s^{(2)}_{2,0},$$
which is, again, equal to (\ref{eqn:quiver_pol_example}).

\section{Proof of Theorem \ref{thm:main}} \label{sec:proof}

{\sl Notation.} Let $\E^e \to X$ be a vector bundle over $X$ (the upper index is optional, it shows the rank of the bundle). The pull-back of this bundle along a map $\rho:Y\to X$ is denoted by $\rho^*\E \to Y$, or simply by $\E \to Y$. The class $c_n(\E)$ may mean the $n$th Chern class of $\E\to X$ or the $n$th Chern class of $\E\to Y$, depending on context.

\subsection{Gysin map in iterated residue form}

The usual description of Gysin maps of Grassmannian (or flag) bundles is either via equivariant localization, or divided difference operations, or the combinatorics of partitions. The idea of studying these Gysin maps in iterated residue form is due to B\'erczi and Szenes \cite[Prop.6.4]{bsz}.
The formulas in their natural generality are developed by Kazarian \cite{kaza:gysin}. The following lemma is a special case of \cite{kaza:gysin}. However, this particular statement is easily seen to be equivalent to the usual combinatorial description of the Gysin map $\pi_*$, see e.g. \cite[Prop. in Sect. 4.1]{fulton_pragacz}.

\begin{lemma}\label{lem:kaz_gys}
Let $\E_1^{e_1}$ and $\E_2^{e_2}$ be vector bundles over $X$, and consider the Grassmannization $\pi:\Gr_{e_2-q}\E_2\to X$ of the bundle $\E_2$ (its fiber is $\Gr_{e_2-q}\C^{e_2}$, the Grassmannian of $e_2-q$-dimensional linear subspaces of $\C^{e_2}$). Let $S^{e_2-q}\to \E_2\to Q^q$ be the tautological exact sequence of bundles over $\Gr_{e_2-q}\E_2$. Let $\omega_1,\ldots,\omega_q$ be the Chern roots of $Q$, and let $\epsilon_1, \ldots, \epsilon_{e_1}$ be the Chern roots of $\E_1$. Let $f$ be a symmetric polynomial in $q$ variables. Then
\[
\pi_*\left( f(\omega_1,\ldots,\omega_q)\prod_{i=1}^q \prod_{j=1}^{e_1} (\omega_i-\epsilon_j) \right)=
\DDelta_{\{v_1,\ldots,v_q\}}^{c_n(\E_1^*-\E_2^*)}\left( f(v_1,\ldots,v_q) \prod_{j=1}^q v_j^{e_1-e_2+q}\right).
\]
\end{lemma}

\subsection{One step resolution in iterated residue form}

Recall that $\QQ=(\QQ_0,\QQ_1)$ is an oriented graph. Consider a $\QQ$-bundle over the space $X$, by which we now mean not only a collection of vector bundles $\E_j^{e_j}$ for $j\in \QQ_0$, but also a vector bundle map $\phi_a:\E_{t(a)} \to \E_{h(a)}$ for every $a\in \QQ_1$.

Let us choose $i\in \QQ_0$ and $q\in \{0,1,\ldots,e_i\}$. The numbers $i$ and $q$ will be fixed for the rest of the section. Following Reineke \cite{reineke} (see also \cite{porteous}, \cite{buch}) we will describe a construction
$$(\E_i, \phi_a)\to X \qquad \qquad \xrightarrow{\hspace*{.5cm} R_i^q \hspace*{.5cm}} \qquad\qquad (\tilde{\E}_i, \tilde{\phi}_a)\to \tilde{X}.$$

Let $\pi:\Gr_{e_i-q}\E_i \to X$ be the Grassmanization of the bundle $\E_i$. Over this space there is the tautological exact sequence $0\to S \to \E_i \to Q \to 0$. Recall that $\M_i=\oplus_{j\in T(i)}\E_j$. Consider the zero scheme $\tilde{X}$ of the composition map $\M_i\to\E_i\to Q$, and its embedding $\iota:\tilde{X} \to \Gr_{e_i-q}\E_i$. The $\QQ$-bundle $(\E_i,\phi_a)$ over $X$ induces a $\QQ$-bundle $(\tilde{\E}_i,\tilde{\phi}_a)$ over $\tilde{X}$ as follows:
\begin{itemize}
\item $\tilde{\E_j}=\E_j$ if $j\not= i$, and $\tilde{\E_i}=S$;
\item $\tilde{\phi}_a=\phi_a$ if $h(a)\not= i$, $t(a)\not= i$;
\item $\tilde{\phi}_a$ is the composition $S \to \E_{i} \xrightarrow{\phi_a} \E_{h(a)}$ if $t(a)=i$;
\item $\tilde{\phi}_a: \E_j \to S$ is the unique lifting of $\phi_a: \E_j\to \E_i$ over $\tilde{X}$, if $h(a)=i$.
\end{itemize}
Let $\rho=\pi \circ \iota: \tilde{X} \to X$.

\smallskip

Our goal now is a formula for the Gysin map $\rho_*$.
Consider $p$ sets of variables
$$\A_1=\{u_{11},\ldots,u_{1r_1}\}, \A_2=\{u_{21},\ldots,u_{2r_2}\}, \ldots, \A_p=\{u_{p1},\ldots,u_{pr_p}\},$$
and the numbers $i_1,i_2,\ldots,i_p \in \QQ_0$. Let $\K_j=(c_n(\M_j^*-\E_j^*))$.
Let $\LL_j$ be obtained from $\K_j$ by replacing $\E_i$ with $S$. (In particular, if $j\not= i$, $j\not\in H(i)$ then $\LL_j=\K_j$.)
Let $\A_0=\{v_{1},\ldots, v_{q} \}$ and $i_0=i$.

\begin{proposition} \label{prop:push_forward}
For the Gysin map $\rho_*: H^*(\tilde{X}) \to H^*(X)$ we have
\[
\rho_*\left(
\DDelta_{\A_1,\ldots,\A_p}^{\K_{i_1},\ldots,\K_{i_p}}\left( f(u_{ks})  \right)\right) =\ \hskip 13 true cm
\]
\[\ \hskip 2 true cm \DDelta_{\A_0,\A_1,\ldots,\A_p}^{\LL_{i_0},\LL_{i_1},\ldots,\LL_{i_p}}\left( f(u_{ks})
\prod_{y \in \A_0} v_j^{\rk \M_i-\rk S}\cdot
\prod_{v \in \A_0} \frac{   \prod_{i_k=i} \hskip .5 true cm\prod_{x \in \A_k} \left(1- \frac{v}{x}\right)}
{   \prod_{i_k\in H(i)} \prod_{x\in \A_k} \left(1- \frac{v}{x}\right)}
  \right).
\]
\end{proposition}

\begin{proof}
Consider
$$g(u_{ks})=f(u_{ks}) \prod_{k=1}^p   \prod_{1\leq l<j \leq r_k} \left( 1-\frac{u_{kl}}{u_{kj}}\right) .$$
Using Lemmas \ref{lemma:C vs Delta} and \ref{lemma:C_vs_const} we can write
\begin{equation}\label{eqn:rand}
z=\DDelta_{\A_1,\ldots,\A_p}^{\K_{i_1},\ldots,\K_{i_p}}\left( f(u_{ks})  \right)=
\left\{
g(u_{ks}) \cdot
\prod_{j=1}^N
\underbrace{\prod_{i_k=j} \prod_{x\in \A_k}  \sum_{n=0}^\infty \frac{c_n(\tilde{\M}_j^*-\tilde{\E}_j^*)}{x^n}}_{\tilde{P}_j}
\right\}_{u_{ks}}
\end{equation}
where  $\{...\}_{u_{ks}}$ means the constant term in the $u_{ks}$ variables. Let
$$P_j=\prod_{i_k=j} \prod_{x\in \A_k}  \sum_{n=0}^\infty \frac{c_n({\M}_j^*-{\E}_j^*)}{x^n} \qquad
\tilde{P_j}=\prod_{i_k=j} \prod_{x\in \A_k}  \sum_{n=0}^\infty \frac{c_n(\tilde{\M}_j^*-\tilde{\E}_j^*)}{x^n}.$$
The appearance of the latter we already indicated in (\ref{eqn:rand}).

We will study the relation between $P_j$ and $\tilde{P}_j$ for various $j$'s. Along the way we will use the formal identity
\[
\sum_{n=0}^\infty \frac{ c_n(\E^*-\F^*) }{x^n} = \frac{  \prod_l (1-\frac{\epsilon_l}{x}) }{  \prod_l (1-\frac{\phi_l}{x}) },
\]
for the Chern roots $\epsilon_l$ (resp. $\phi_l$) of the bundle $\E$ (resp. $\F$).
\begin{itemize}
\item If $j\not=i$, $j\not\in H(i)$ then obviously $\tilde{P}_j=P_j$.
\item Let the Chern roots of $\tilde{\M}_i=\M_i$ be $\mu_l$, let the Chern roots of $\tilde{\E}_i=S$ be $\sigma_l$, let the Chern roots of $\E_i$ be $\beta_l$, and let the Chern roots of $Q$ be $\omega_l$. Then
    $$\tilde{P}_i=\prod_{i_k=i} \prod_{x\in \A_k}  \frac{  \prod_l \left( 1- \frac{\mu_l}{x} \right) }{  \prod_l \left( 1- \frac{\sigma_l}{x} \right) }=\prod_{i_k=i} \prod_{x\in \A_k} \frac{  \prod_l \left( 1- \frac{\mu_l}{x} \right) \prod_l \left( 1- \frac{\omega_l}{x} \right)}{  \prod_l \left( 1- \frac{\beta_l}{x} \right) }=
    P_i\prod_{i_k=i} \prod_{x\in \A_k}\prod_l \left( 1- \frac{\omega_l}{x} \right).$$
\item Let $j\in H(i)$. Let the Chern roots of $\tilde{M}_j$ be $\nu_l$, the Chern roots of $\E_j$ be $\kappa_l$, and the Chern roots of $\M_j$ be $\gamma_l$. We have $$\tilde{P}_j=\prod_{i_k=j} \prod_{x\in \A_k}  \frac{  \prod_l \left( 1- \frac{\nu_l}{x} \right) }{  \prod_l \left( 1- \frac{\kappa_l}{x} \right) }=\prod_{i_k=j} \prod_{x\in \A_k} \frac{  \prod_l \left( 1- \frac{\gamma_l}{x} \right) }{  \prod_l \left( 1- \frac{\kappa_l}{x} \right) \prod_l \left( 1- \frac{\omega_l}{x} \right) }=
    P_j\prod_{i_k=j} \prod_{x\in \A_k}\prod_l \frac{1}{\left( 1- \frac{\omega_l}{x} \right)}.$$
\end{itemize}

Thus (\ref{eqn:rand}) can be written as
\begin{equation} \label{eqn:tt1}
z=\left\{ g(u_{ks})
\frac{\prod_{i_k=i} \prod_{x\in \A_k} \prod_l \left(1- \frac{\omega_l}{x}\right)}{\prod_{i_k\in H(i)} \prod_{x\in \A_k} \prod_l \left(1- \frac{\omega_l}{x}\right)} \prod_{j=1}^N P_j
\right\}_{u_{ks}}.
\end{equation}
This class is an expression in Chern roots of bundles, all of which pull back via $\iota$. Hence, the adjunction formula
$\iota_*(\iota^*(h))=h\cdot \iota_*(1)$ of the Gysin map implies that $\iota_*(z)$ is the same expression as (\ref{eqn:tt1}) times
$$\iota_*(1)=e(\Hom(\M_i,Q))=\prod_{l=1}^{q}\prod_{j=1}^{\rk \M_i} (\omega_l-\mu_j).$$
We can apply Lemma \ref{lem:kaz_gys} to obtain $\rho_*(z)=\pi_*\iota_*(z)=$
\[
\DDelta_{\A_0}^{\K_{i_0}}\left(
\prod_{j=1}^q v_j^{\rk \M_i -\rk S}
 \left\{
g(u_{ks})
\frac{\prod_{i_k=i} \prod_{x\in \A_k} \prod_l \left(1- \frac{v_l}{x}\right)}{\prod_{i_k\in H(i)} \prod_{x\in \A_k} \prod_l \left(1- \frac{v_l}{x}\right)} \prod_{j=1}^N P_j
\right\}_{u_{ks}}
\right).
\]
This is further equal to
\[
\left\{
\prod_{j=1}^q v_j^{\rk \M_i -\rk S+1-j}
\prod_{1\leq l < j \leq q} (v_j-v_l) \cdot g(u_{ks})
\frac{\prod_{i_k=i} \prod_{x\in \A_k} \prod_l \left(1- \frac{v_l}{x}\right)}{\prod_{i_k\in H(i)} \prod_{x\in \A_k} \prod_l \left(1- \frac{v_l}{x}\right)} \prod_{j=1}^N P_j
\right\}_{u_{ks},v_j}
\]
which is---using Lemmas \ref{lemma:C vs Delta} and \ref{lemma:C_vs_const}---equal to
\[
\DDelta_{\A_0,\A_1,\ldots,\A_p}^{\LL_{i_0},\LL_{i_1},\ldots,\LL_{i_p}}\left( f(u_{ks})
\prod_{y \in \A_0} v_j^{\rk \M_i-\rk S}\cdot
\prod_{v \in \A_0} \frac{   \prod_{i_k=i} \hskip .5 true cm\prod_{x \in \A_k} \left(1- \frac{v}{x}\right)}
{   \prod_{i_k\in H(i)} \prod_{x\in \A_k} \left(1- \frac{v}{x}\right)}
  \right).
\]
This concludes the proof.
\end{proof}

\subsection{Iterated application of Proposition \ref{prop:push_forward}}

Now we are ready to prove Theorem \ref{thm:main}.
Let $\QQ$ be a quiver of Dynkin type, and let $\Omega_m$ be an orbit of a quiver representation. Let $\ii,\rr$ be a resolution pair for $\Omega_m$.

Let $(\E_i, \phi_a)\to X$ be a $\QQ$-bundle over the space $X$. Define the degeneracy locus
$$\Omega(\E_i,\phi_i) = \{x\in X: (E_i,\phi_a)_x\in \overline{\Omega}_m\}.$$

\begin{theorem} (Reineke \cite[Thm. 2.2]{reineke}) \label{thm:reineke}
Let $\tilde{X}$ be the base space of the bundle obtained by applying $R_{i_1}^{r_1}$, then $R_{i_1}^{r_1}$, etc, up to $R_{i_p}^{r_p}$ to the $\QQ$-bundle $(\E_i, \phi_a)\to X$. The map
$$\rho_{i_1}^{r_1} \circ \rho_{i_2}^{r_2} \circ \ldots \circ \rho_{i_p}^{r_p} : \tilde{X} \to X$$
is a resolution of $\overline{\Omega}(\E_i,\phi_i)$.
\end{theorem}


We can apply Theorem \ref{thm:reineke} to maps between the universal vector bundles described in Section~\ref{sec:poly}, and we get
\[
[\overline{\Omega}_m]=\left( \rho_{i_1}^{r_1} \circ \rho_{i_2}^{r_2} \circ \ldots \circ \rho_{i_p}^{r_p} \right)_* (1)
= {\rho_{i_1}^{r_1}}_* \left( {\rho_{i_2}^{r_2}}_* \left( \ldots \left( {\rho_{i_p}^{r_p}}_* \left( \DDelta_{\emptyset}^{\emptyset}(1) \right)\right)\ldots\right)\right).
\]
If we apply the statement of Theorem \ref{prop:push_forward} at each application of $\rho_*$, we get the statement of Theorem \ref{thm:main}.

\section{On Buch's conjecture}

For an alphabet $\X=(x_1,\ldots,x_r)$ and a sequence of integers $\lambda \in \Z^r$ let $\X^\lambda$ be $\prod_{x_i \in \X} x_i^{\lambda(i)}$. We can reformulate Buch's conjecture on the positivity of cohomological quiver coefficients, in the terminology of the present paper.

\begin{conjecture} (Buch \cite{buch}) For Dynkin quivers we have
\begin{equation}\label{eqn:buchform}
[\overline{\Omega}_m]=\DDelta_{\X_1,\ldots,\X_N}^{\K_1,\ldots,\K_N}\left(
f(\X_1,\ldots,\X_N)
\right),
\end{equation}
where
$$f=\sum c_{(\lambda_1,\ldots,\lambda_N)} \X^{\lambda_1}\cdots \X^{\lambda_N}$$
with all $\lambda_i$ being partitions (i.e. weakly decreasing sequences) and all $c_{(\lambda_1,\ldots,\lambda_N)} \geq 0$.
\end{conjecture}

In (\ref{eqn:main3}) we gave an $f$ satisfying (\ref{eqn:buchform}). However, that particular $f$ may not satisfy the positivity condition. The challenge in proving Buch's conjecture is to find an $f$ ``$\DDelta$-equivalent'' to the formula in (\ref{eqn:main3}) but satisfying the positivity conditions. Initial results in this direction are in \cite{ryan_phd}.

\end{document}